\documentclass{amsart}
\usepackage{latexsym, graphicx, amsmath, amssymb, cite, hyperref}

\newtheorem{theorem}{Theorem}[section]
\newtheorem{lemma}[theorem]{Lemma}
\newtheorem{proposition}[theorem]{Proposition}

\newtheorem{conjecture}[theorem]{Conjecture}

\theoremstyle{definition}

\newtheorem{example}[theorem]{Example}

\theoremstyle{remark}
\newtheorem{remark}[theorem]{Remark}

\newcommand{\R}{{\mathbb R}}
\newcommand{\Z}{{\mathbb Z}}
\newcommand{\N}{{\mathbb N}}
\newcommand{\T}{{\mathbb T}}
\newcommand{\supp}{\operatorname{supp}}
\newcommand{\vol}{\operatorname{vol}}
\renewcommand{\div}{\operatorname{div}}

\newcommand{\Gr}{\operatorname{Gr}}

\title{Weyl Law Improvement for Products of Spheres}
\author{Alex Iosevich and Emmett Wyman}
\date{September 26, 2019}

\begin{document}

\maketitle

\begin{abstract} The classical Weyl Law says that if $N_M(\lambda)$ denotes the number of eigenvalues of the Laplace operator on a $d$-dimensional compact manifold $M$ without a boundary that are less than or equal to $\lambda$, then 
$$ N_M(\lambda)=c\lambda^d+O(\lambda^{d-1}).$$ 

In this paper, we show Duistermaat and Guillemin's result allows us to replace the $O(\lambda^{d-1})$ error with $o(\lambda^{d-1})$ if $M$ is a product manifold. We quantify this bound in the case of Cartesian product of spheres by reducing the problem to the study of the distribution of weighted integer lattice points in Euclidean space and formulate a conjecture in the general case.
\end{abstract}


\section{Introduction}

Let $M$ be a compact Riemannian manifold without boundary with Laplace-Beltrami operator $\Delta_M$. By the spectral theorem, we write
\[
    L^2(M) = \operatorname{cl}\left( \bigoplus_{\lambda \in \Lambda} E_\lambda \right),
\]
where
\[
    \Lambda = \{\lambda \in [0,\infty) : -\lambda^2 \text{ is an eigenvalue of $\Delta$}\}
\]
and $E_\lambda$ is the eigenspace corresponding to the eigenvalue $-\lambda^2$. The dimension of $E_\lambda$ is called the \emph{multiplicity} of $\lambda$ and will be denoted $\mu(\lambda)$. The Weyl counting function is given by
\[
    N(\lambda) = \sum_{\lambda' \in \Lambda \cap [0,\lambda]} m(\lambda')
\]
and satisfies the \emph{Weyl law},
\begin{equation}\label{the weyl law}
    N(\lambda) = \frac{|B_d| \vol M}{(2\pi)^d} \lambda^{d} + O(\lambda^{d - 1})
\end{equation}
where $d$ is the dimension of $M$ and $|B_d|$ denotes the volume of the unit ball in $\R^d$. See, for example, \cite{Stern99} and the references contained therein. The big-$O$ remainder in the Weyl law is sharp for some manifolds, the standard example being the sphere, whose spectrum is well-known.

\begin{example}[The Standard Sphere]
The spectrum of the Laplace-Beltrami operator on $S^d$ is given by
\[
    \Lambda = \left\{\sqrt{k(k + d - 1)} : k \in \Z_{\geq 0} \right\}
\]
with corresponding multiplicities
\begin{equation}\label{sphere spectrum}
    \mu(\sqrt{k(k + d - 1)}) = \binom{d+k}{d} - \binom{d+k-2}{d} = \frac{2}{(d-1)!} k^{d-1} + O(k^{d-2}).
\end{equation}
Hence, the Weyl counting function has jumps of order $\lambda^{d-1}$ which saturates the standard remainder term. (See \cite[Section 3.4]{Hang} or the introduction of \cite{helgason}.)
\end{example}

The big-$O$ remainder term in \eqref{the weyl law} may be improved under various assumptions on the manifold, e.g. a qualitative improvement can be obtained if the set of closed geodesics is sufficiently sparse \cite{DG,Iv}, and an improvement by $\log \lambda$ can be gained if the sectional curvature is nonpositive \cite{Berard}. In recent work, Canzani and Galkowski obtain, among a large number of other results, a $\log \lambda$ improvement to the remainder under some dynamical assumptions on the geodesic flow \cite{CGbeams}.

The theorem of \cite{DG} is fundamental, and we summarize it below. For a Riemannian manifold $M$ with metric $g$, let $p$ be the principal symbol of the pseudo-differential operator $\sqrt{-\Delta_g}$, expressed in the canonical local coordinates 
$$(x_1,\ldots, x_n, \xi_1,\ldots, \xi_n) \ \text{of} \ T^*M$$ by
\begin{equation} \label{principal symbol}
    p(x,\xi) = \left( \sum_{i,j = 1}^n g_{ij}(x) \xi_i \xi_j \right)^{1/2}.
\end{equation}
The \emph{Hamilton vector field} on $T^*M$ associated to the symbol $p$ is
\begin{equation} \label{Hamiltonian vector field}
    H_p = \sum_{i = 1}^n \frac{\partial p}{\partial \xi_i} \frac{\partial}{\partial x_i} - \frac{\partial p}{\partial x_i} \frac{\partial}{\partial \xi_i}.
\end{equation}
The flow associated with the Hamilton vector field $H_p$ is called the \emph{Hamilton flow}.

\begin{theorem}[Duistermaat and Guillemin \cite{DG} and Ivrii \cite{Iv}] \label{Duistermaat Guillemin}
    Let $M$ be a $d$-dimensional, compact Riemannian manifold without boundary and let $N$ be the Weyl counting function associated to the Laplace-Beltrami operator on $M$. Let $\Sigma \subset S^*M$ denote the set of unit covectors which belong to periodic orbits of the Hamilton flow. If $\Sigma$ has measure zero in $S^*M$,
    \[
        N(\lambda) = \frac{|B_d| \vol M}{(2\pi)^d} \lambda^{d} + o(\lambda^{d - 1}).
    \]
\end{theorem}

\vskip.125in 

\begin{remark} For the symbol $p$ as above, a smooth curve $\gamma$ in $M$ is a geodesic if and only if $t \mapsto g(\gamma'(t), \cdot ) \in T^*M$ is an integral curve of $H_p$. In this sense, the bicharacteristics of $H_p$ are typically identified with geodesics on $M$ and the Hamiltonian flow on $S^*M$ is typically identified with the geodesic flow on $SM$.\end{remark} 

The geodesic flow on product manifolds gives rise to a suitably thin set of periodic orbits as required by Theorem \ref{Duistermaat Guillemin}. Our first result is the following. 

\begin{theorem} \label{DG products}
    Let $M_1, \ldots, M_n$ be compact Riemannian manifolds, without boundary, with $d_i = \dim M_i \geq 1$ for $i = 1,\ldots,n$ and $n \geq 2$. The set of unit covectors in $M$ belonging to periodic orbits of the Hamiltonian flow has measure zero, and hence by Theorem \ref{Duistermaat Guillemin},
    \[
        N(\lambda) = \frac{|B_d|\vol(M)}{(2\pi)^{|d|}} \lambda^{|d|} + o(\lambda^{|d|-1}),
    \]
    where $N$ is the Weyl counting function for the product manifold $M = M_1 \times \cdots \times M_n$ and $|d| = \dim(M) = d_1 + \cdots + d_n$.
\end{theorem}

The main purpose of this paper is to quantify the gain in the $o(\lambda^{|d|-1})$ term in Theorem \ref{DG products}. The main result of this paper is the following. 

\begin{theorem}\label{weyl}
    Let $M = S^{d_1} \times \cdots \times S^{d_n}$ be a product of $n$ spheres of dimension $d_i \geq 1$ for $i = 1,\ldots,n$. Let $\Delta$ be the Laplace-Beltrami operator with respect to the product metric on $M$. If $-\lambda_j^2$ for $j = 1,2,\ldots$ are the eigenvalues of $\Delta$ repeated with multiplicity, then we have bounds
    \[
        N(\lambda) := \#\{ j : \lambda_j \leq \lambda \} = \frac{|B_{|d|}|}{(2\pi)^{|d|}} \vol(M) \lambda^{|d|} + O(\lambda^{|d|-1-\frac{n-1}{n+1}}).
    \]
    where $|d|$ denotes $\dim M = d_1 + \cdots + d_n$.
\end{theorem}

It is natural to ask if the conclusion of Theorem \ref{DG products} still holds if we replace the spectra by an arbitrary Weyl law distributed set of points in $\R$. The following proposition demonstrates that the improvement in the theorem is not a combinatorial effect of the real numbers, but an effect of the product structure of the (smooth) manifold.

\begin{proposition} \label{not a spectrum}
There exists a discrete subset $\Lambda$ of the positive reals such that each unit interval $[n,n+1)$ for $n = 0,1,2,\ldots$ contains exactly one point in $\Lambda$ and
\[
    \liminf_{\lambda \to \infty} \frac{|\# (\Lambda \times \Lambda) \cap \lambda B_2 - \pi \lambda^{2}|}{\lambda} > 0.
\]
\end{proposition}

The flat torus $\T^d = \R^d/2\pi \Z^d$ is a model example of a manifold possessing a thin set of closed geodesics and a better remainder term than in \eqref{the weyl law}.

\begin{example}[The Flat Torus]
The flat torus admits a Hilbert basis of eigenfunctions in the form of exponentials
\[
    e_m(x) = (2\pi)^{-d/2} e^{-i\langle x , m \rangle} \qquad \text{ for $m \in \Z^d$}
\]
each satisfying
\[
    \Delta e_m = -|m|^2 e_m.
\]
Hence we have spectrum and multiplicities,
\[
    \Lambda = \{ |m| : m \in \Z^d \} \qquad \text{ and } \qquad \mu(\lambda) = \#\{ m : |m| = \lambda\}.
\]
The Weyl counting function then just counts the number of integer lattice points lying in the ball of radius $\lambda$,
\begin{equation}\label{torus counting function}
    N(\lambda) = \#\{m \in \Z^d : |m| \leq \lambda\}.
\end{equation}
Note the set of closed geodesics are precisely the lines with rational slope, and hence constitute a set of measure zero in the cotangent bundle $T^* \T^d$. By Theorem \ref{Duistermaat Guillemin}, we automatically have a little-$o$ improvement to the Weyl remainder term in \eqref{the weyl law}. However, it was known that $N(\lambda)$ satisfies much better remainder bounds long before these microlocal results.

It was proved in 1950 by Hlawka \cite{Hlawka1950}, using the Poisson summation formula that 
\[
    N(\lambda) = |B_d| \lambda^{d} + O(\lambda^{d-1 - \frac{d-1}{d+1}}).
\]
The remainder bound is not sharp and the exponent has been improved little by little over the decades. In dimensions $d \geq 5$, there is a sharp remainder of $O(\lambda^{d-2})$ (see e.g. \cite{Fricker1982}). In dimension $d = 2$, the conjectured $O(\lambda^{\frac12 + \epsilon})$ remainder for all $\epsilon > 0$ remains open, with the best current exponent due to Bourgain and Watt \cite{BourgainWatt2017}. In three dimensions, the conjectured $O(\lambda^{1+\epsilon})$ bound is also open, with the best known exponent due to Heath-Brown \cite{Heath-Brown2000}. See also \cite{Huxley1996, KratzelNowak1992} for a thorough description of the distribution of lattice points in the ball, and, consequently, the Weyl law results on the torus.
\end{example}

The aforementioned sharp $O(\lambda^{d-2})$ remainder for the flat torus of dimension $d \geq 5$ provides an obstruction for improvements to the remainder term in Theorem \ref{DG products}. In the case of general manifolds, we conjecture the following. 

\vskip.125in 

\begin{conjecture} Let $M_1, \dots, M_k$ denote compact Riemannian manifolds of dimension $d_1, d_2, \dots, d_n>0$.  Let $\Delta$ be the Laplace-Beltrami operator with respect to the product metric on $M$. If $-\lambda_j^2$ for $j = 1,2,\ldots$ are the eigenvalues of $\Delta$ repeated with multiplicity, then we have bounds
    \[
        N(\lambda) := \#\{ j : \lambda_j \leq \lambda \} = \frac{|B_{|d|}|}{(2\pi)^{|d|}} \vol(M) \lambda^{|d|} + O(\lambda^{|d|-1-\delta}).
    \]
    for some $\delta>0$, where $|d|$ denotes $\dim M = d_1 + \cdots + d_n$.
\end{conjecture}
    
    If $n \ge 5$, then we believe that we may take $\delta=1$. The case when each $M_j$ is a torus, discussed in the context of lattice point distribution results above, shows that this gain would be best possible. We hope to address this issue in the sequel. 
    
\medskip

The proof of Theorem \ref{weyl} amounts to a reduction to the following weighted lattice point problem.

\begin{theorem} \label{lattice}
    Let $\R_+$ denote the nonnegative real numbers and let $y \in \R^n$. Consider a multi-index $d = (d_1, d_2, \ldots, d_k, 1,\ldots,1) \in \N^n$ with $d_i \geq 2$ for $i = 1,\ldots, k$. Then,
    \begin{multline*}
        \sum_{\substack{m \in (\Z^n + y) \cap \R^k_+ \times\R^{n-k} \\ |m| \leq \lambda}} m_1^{d_1-1} \cdots m_k^{d_k - 1} = \lambda^{|d|} \int_{B \cap \R^k_+ \times \R^{n-k}} x_1^{d_1-1} \cdots x_k^{d_k-1} \, dx + E(\lambda)
    \end{multline*}
    where $E(\lambda)$ satisfies bounds
    \begin{equation}\label{remainder bound}
        E(\lambda) = O(\lambda^{|d|-1 - \frac{n-1}{n+1}})
    \end{equation}
    uniformly in $y$.
\end{theorem}

\begin{remark}
	Inspection of the proof of Theorem \ref{weyl} shows that any improvement to the exponent in \eqref{remainder bound} directly transfers to the same improvement to the exponent in the remainder bound of Theorem \ref{weyl}, up to a minimum exponent of $|d|-2$.
\end{remark}

We prove Theorem \ref{lattice} using the standard strategy with some minor modifications. We mollify the product of the characteristic function of a ball of radius $\lambda$ and the homogeneous weight function. We compute bounds on the Fourier transform of this mollified product and conclude the argument using the Poisson summation formula.


\section{Proof of Theorem \ref{weyl}} \label{Section 2}

Let $M_1$ and $M_2$ both be compact Riemannian manifolds without boundary with $d_i = \dim M_i$ for $i = 1,2$. Their product $M = M_1 \times M_2$ is again a boundaryless, compact Riemannian manifold endowed with the product metric. The Laplace-Beltrami operator on $M$ is
\[
    \Delta_M = \Delta_{M_1} \otimes I + I \otimes \Delta_{M_2}.
\]
If $e_1$ and $e_2$ are eigenfunctions on $M_1$ and $M_2$, respectively, with
\[
    \Delta_{M_i} e_i = -\lambda_i^2 e_i \qquad \text{ for $i = 1,2$},
\]
then their tensor $e_1 \otimes e_2$ is an eigenfunction of the Laplacian $\Delta_M$ with
\[
    \Delta_M e_1 \otimes e_2 = (\Delta_{M_1} e_1) \otimes e_2 + e_1 \otimes (\Delta_{M_2} e_2) = -(\lambda_1^2 + \lambda_2^2)e_1 \otimes e_2.
\]
If $e_1$ and $e_2$ are drawn from a Hilbert basis of eigenfunctions on $M_1$ and $M_2$, respectively, then the tensors $e_1 \otimes e_2$ form a Hilbert basis for $L^2(M)$. We can construct the spectrum $\Lambda$ on $M$ from the spectra $\Lambda_1$ and $\Lambda_2$ for $M_1$ and $M_2$ by
\[
    \Lambda = \left\{\sqrt{\lambda_1^2 + \lambda_2^2} : \lambda_1 \in \Lambda_1 \text{ and } \lambda_2 \in \Lambda_2 \right\}
\]
with multiplicities
\[
    \mu(\lambda) = \sum_{\substack{(\lambda_1,\lambda_2) \in \Lambda_1 \times \Lambda_2 \\ \lambda_1^2 + \lambda_2^2 = \lambda^2}} \mu_1(\lambda_1) \mu_2(\lambda_2),
\]
where here $\mu_1$ and $\mu_2$ are the respective multiplicities for $\Lambda_1$ and $\Lambda_2$. The Weyl counting function for $M$ can be written
\[
    N(\lambda) = \sum_{\substack{(\lambda_1,\lambda_2) \in \Lambda_1 \times \Lambda_2 \\ \lambda_1^2 + \lambda_2^2 \leq \lambda^2}} \mu_1(\lambda_1) \mu_2(\lambda_2).
\]
A similar formula holds for the Weyl counting function if $M$ is an $n$-fold product $M_1 \times \cdots \times M_n$ of compact, boundaryless Riemannian manifolds with respective spectra and multiplicities $\Lambda_i$ and $\mu_i$ for each $i$. Namely,
\begin{equation}\label{product counting formula}
    N(\lambda) = \sum_{\substack{(\lambda_1,\ldots,\lambda_n) \in \Lambda_1 \times \cdots \times \Lambda_n \\ \lambda_1^2 + \cdots + \lambda_n^2 \leq \lambda^2}} \prod_{i = 1}^n \mu_i(\lambda_i).
\end{equation}

We specify to the case $M = S^{d_1} \times S^{d_2} \times \cdots \times S^{d_n}$. If $M$ contains any $S^1$ factors, we gather them on the rightmost side of the product. That is,
\[
    M = S^{d_1} \times \cdots \times S^{d_k} \times \underbrace{S^1 \times \cdots \times S^1}_{\text{$n-k$ times}} \simeq S^{d_1} \times \cdots \times S^{d_k} \times \T^{n-k}
\]
where $\T^{n-k} = \R^{n-k}/2\pi \Z^{n-k}$ is the $(n-k)$-dimensional torus. We also write the dimension multiindex
\[
    d = (d_1,\ldots,d_k,\underbrace{1,\ldots,1}_{\text{$n-k$ times}}) \qquad \text{ with } d_1,\ldots, d_k \geq 2.
\]
The Weyl counting function for $M$ is
\begin{equation} \label{sphere product counting function}
    N(\lambda) = \sum_{\substack{m \in \Z_{\geq 0}^k \times \Z^{d-k} \\ |m + y|^2 \leq \lambda^2 + |y|^2}} \prod_{i = 1}^k \left( \binom{m_i + d_i}{d_i} - \binom{m_i + d_i - 2}{d_i} \right)
\end{equation}
by \eqref{sphere spectrum}, \eqref{torus counting function}, and \eqref{product counting formula}, where
\[
    y = \left( \frac{d_1 - 1}{2} , \ldots , \frac{d_k - 1}{2} , 0, \ldots, 0 \right).
\]
For each $i$, we the multiplicity $\binom{m_i}{d_i} - \binom{m_i - 2}{d_i}$ in the product above coincides with the polynomial
\begin{equation*}
    P_{d_i-1}\left(m_i + \frac{d_i-1}{2}\right) = \frac{2}{(d_i-1)!}\left(m_i + \frac{d_i-1}{2}\right)\frac{(m_i+d_i-2)!}{m_i!} \qquad \text{ if $m_i \geq 2$.}
\end{equation*}
Note $P_{d_i-1}$ is a polynomial of degree $d_i-1$. Moreover the zeroes of $P_{d_i-1}$ are distributed symmetrically about $0$, and hence $P_{d_i-1}$ is either even or odd and
\begin{equation} \label{polynomial}
    P_{d_i-1}(t) = \frac{2}{(d_i-1)!} t^{d_i-1} + O(t^{d_i-3}) \qquad \text{ for $|t|$ large.}
\end{equation}
From here we make a few reductions at the loss of negligible  $O(\lambda^{|d|-2})$ terms. The following lemma both eliminates the contributions of all but the leading terms of $P_{d_i-1}$ and makes a convenient change of variables in the sum.

\begin{lemma} \label{reduction} Let $N(\lambda)$ be the Weyl counting function for $M$. Then,
\[
    N(\lambda) = \sum_{m \in (\Z^n + y) \cap \R_+^k \times \R^{n-k}} \chi_B(|m|/\sqrt{\lambda^2 + |y|^2}) \prod_{i = 1}^k \frac{2}{(d_i-1)!} m_i^{d_i - 1} + O(\lambda^{|d|-2}).
\]
\end{lemma}

\begin{proof}
    We use $\Z_{\geq 2}$ to denote the set of integers greater than or equal to $2$ and let
    \[
        N_1(\lambda) = \sum_{\substack{m \in \Z_{\geq 2}^k \times \Z^{n-k} \\ |m + y|^2 \leq \lambda^2 + |y|^2}} \prod_{i = 1}^k \left( \binom{m_i + d_i}{d_i} - \binom{m_i + d_i - 2}{d_i} \right),
    \]
    which is equal to $N(\lambda)$ save for the exclusion of integer lattice points with a $0$ or $1$ in the first $k$ coordinates. For $m_i = 0,1$, we have
    \[
        \binom{m_i + d_i}{d_i} - \binom{m_i + d_i - 2}{d_i} = \begin{cases}
            1 & m_i = 0 \\
            d_i + 1 & m_i = 1
        \end{cases}
    \]
    and so
    \begin{equation} \label{reduction 0}
        N(\lambda) - N_1(\lambda)
        = \begin{cases}
        O(\lambda^{n-1}) & \text{ if $k \geq 1$} \\
        0 & \text{ if $k = 0$}.
        \end{cases}
    \end{equation}
    Note the right hand side is $O(\lambda^{|d| - 2})$. Indeed, if $k \geq 1$, then $n - 1 \leq |d|-2$. We rewrite $N_1(\lambda)$ as
    \[
        N_1(\lambda) = \sum_{m \in \Z_{\geq 2}^k \times \Z^{n-k} + y} \chi_D(|m|/\sqrt{\lambda^2 + |y|^2}) \prod_{i = 1}^k P_{d_i-1}(m_i)
    \]
    by using the polynomial introduced in \eqref{polynomial} and reindexing the sum. Setting
    \[
        N_2(\lambda) = \sum_{m \in \Z_{\geq 2}^k \times \Z^{n-k} + y} \chi_D(|m|/\sqrt{\lambda^2 + |y|^2}) \prod_{i = 1}^k \frac{2}{(d_i - 1)!} m_i^{d_i-1}
    \]
    and invoking \eqref{polynomial} yields
    \begin{align}
        \nonumber N_1(\lambda) &= \sum_{m \in \Z_{\geq 2}^k \times \Z^{n-k} + y} \chi_D(|m|/\sqrt{\lambda^2 + |y|^2}) \prod_{i = 1}^k \left( \frac{2}{(d_i - 1)!} m_i^{d_i-1} + O(m_i^{d_i - 3}) \right) \\
        \nonumber &= N_2(\lambda) + \sum_{m \in \Z_{\geq 2}^k \times \Z^{n-k} + y} \chi_D(|m|/\sqrt{\lambda^2 + |y|^2}) O(|m|^{|d|-n-2}) \\
        \label{reduction 1} &= N_2(\lambda) + O(\lambda^{|d|-2}),
    \end{align}
    where the final line follows by the naive estimate.
    
    Now we bridge the gap between $N_2(\lambda)$ and the sum in the statement of the lemma, namely
    \[
        N_3(\lambda) = \sum_{m \in (\Z^n + y) \cap \R_+^k \times \R^{n-k}} \chi_B(|m|/\sqrt{\lambda^2 + |y|^2}) \prod_{i = 1}^k \frac{2}{(d_i-1)!} m_i^{d_i - 1}.
    \]
    We write the discrepancy as
    \[
        N_3(\lambda) - N_2(\lambda) = \sum_{m \in (\Z^n + y) \cap \bigcup_{j = 1}^k H_j} \chi_D(|m|/\sqrt{\lambda^2 + |y|^2}) \prod_{i = 1}^k \frac{2}{(d_i - 1)!} m_i^{d_i-1}
    \]
    where $H_j = \{ x \in \R_{\geq 0}^k \times \R^{n-k} : x_j \in [0,y_j + 2) \}$. For each $j = 1,\ldots,k$,
    \[
        \prod_{i = 1}^k \frac{2}{(d_i - 1)!} x_i^{d_i-1} = O(|x|^{|d|-n-d_j+1}) \qquad x \in H_j.
    \]
    Hence,
    \[
        \sum_{m \in (\Z^n + y) \cap H_j} \chi_D(|m|/\sqrt{\lambda^2 + |y|^2}) \prod_{i = 1}^k \frac{2}{(d_i - 1)!} m_i^{d_i-1} = O(\lambda^{|d|-d_j}).
    \]
    Since each $d_j \geq 2$, the bound is at least as good as $O(|x|^{|d|-n-1})$. Hence,
    \begin{equation}\label{reduction 2}
        N_3(\lambda) - N_2(\lambda) = \sum_{j =1}^k O(\lambda^{|d|-d_j}) = O(\lambda^{|d|-2}).
    \end{equation}
    The lemma follows from \eqref{reduction 0}, \eqref{reduction 1}, and \eqref{reduction 2}.
\end{proof}

Taylor expansion yields
\[
    (\lambda^2 + |y|^2)^{|d|/2} = \lambda^{|d|}\left(1 + \frac{|y|^2}{\lambda^2}\right)^{|d|/2} = \lambda^{|d|} + O(\lambda^{|d|-2}).
\]
This along with Lemma \ref{reduction} and Theorem \ref{lattice} implies
\[
    N(\lambda) = C (\lambda^2 + |y|^2)^{|d|/2} + O(\lambda^{|d| - 1 - \frac{n-1}{n+1}}) = C\lambda^{|d|} + O(\lambda^{|d| - 1 - \frac{n-1}{n+1}})
\]
where
\[
    C = \int_{B \cap \R_{+}^k \times \R^{n-k}} \prod_{i=1}^n \frac{2}{(d_i-1)!} x_i^{d_i - 1} \, dx.
\]
This concludes the proof of Theorem \ref{weyl}. One notes that the constant $C$ must necessarily equal $\omega_{|d|} \vol(M)/(2\pi)^{|d|}$, the coefficient of the main term of Theorem \ref{weyl}.


\section{Proof of Theorem \ref{lattice}}

For clarity, set
\[
    F(x) = \prod_{i=1}^k \chi_{[0,\infty)}(x_i) x^{d_i-1}.
\]
The equation in the theorem then reads
\[
    \sum_{m \in \Z^n + y} \chi_{\lambda B}(m) F(m) = \lambda^{|d|} \int_{B} F(x) \, dx + E(\lambda).
\]
Let $N(\lambda)$ denote the left side and $C_d = \int_B F(x) \, dx$ the constant on the right, so that $E(\lambda) = N(\lambda) - C_d \lambda^{|d|}$.

Let $\rho$ be a smooth, nonnegative function supported in $B \subset \R^n$ with $\int_{\R^n} \rho(x) \, dx = 1$. For $\epsilon > 0$, we set $\rho_\epsilon(x) = \epsilon^{-n} \rho(\epsilon^{-1} x)$. Note $\rho_\epsilon$ is supported in the ball of radius $\epsilon$ and $\int_{\R^n} \rho_\epsilon(x) \, dx = 1$. We define a mollified sum
\[
    N_\epsilon(\lambda) = \sum_{m \in Z^n + y} \chi_{\lambda B} * \rho_\epsilon(|m|) F(m) = C_d \lambda^{|d|} + E_\epsilon(\lambda).
\]
Note,
\[
    N_\epsilon(\lambda - \epsilon) \leq N(\lambda) \leq N_\epsilon(\lambda + \epsilon),
\]
and hence
\[
    E_\epsilon(\lambda - \epsilon) - C_d(\lambda^{|d|} - (\lambda - \epsilon)^{|d|}) \leq E(\lambda) \leq E_\epsilon(\lambda + \epsilon) + C_d((\lambda+\epsilon)^{|d|} - \lambda^{|d|}).
\]
Hence we will have
\[
    |E(\lambda)| \lesssim \epsilon^{-\frac{n-1}{2}} \lambda^{|d| - \frac{n+1}{2}} + \epsilon \lambda^{|d|-1}
\]
provided we can show
\begin{equation}\label{E bound 1}
    |E_\epsilon(\lambda)| \lesssim \epsilon^{-\frac{n-1}{2}} \lambda^{|d| - \frac{n+1}{2}} + \epsilon \lambda^{|d|-1}.
\end{equation}
We optimize by setting
\begin{equation}\label{optimize epsilon}
    \epsilon = \lambda^{-\frac{n-1}{n+1}}
\end{equation}
from which we recover the theorem.

Before using the Poisson summation formula, we would like to exchange the order of multiplication and convolution in the sum. That is if
\[
    \tilde N_\epsilon(\lambda) = \sum_{m \in \Z^n + y} (\chi_{\lambda B} F)*\rho_\epsilon(m) = C_d\lambda^{|d|} + \tilde E_\epsilon(\lambda),
\]
we would like
\begin{equation} \label{E discrepancy}
    |E_\epsilon(\lambda) - \tilde E_\epsilon(\lambda)| = |N_\epsilon(\lambda) - \tilde N_\epsilon(\lambda)| = O( \epsilon \lambda^{|d|-1})
\end{equation}
so that \eqref{E bound 1} follows from
\begin{equation}\label{E bound 2}
    |\tilde E_\epsilon(\lambda)| \lesssim \epsilon^{-\frac{n-1}{2}} \lambda^{|d| - \frac{n+1}{2}} + \epsilon \lambda^{|d|-1}.
\end{equation}
The following lemma provides us with \eqref{E discrepancy}.

\begin{lemma}\label{commutator}
\[
    |N_\epsilon(\lambda) - \tilde N_\epsilon(\lambda)| = O(\epsilon \lambda^{|d|-1}).
\]
\end{lemma}

\begin{proof}
    $F$ is Lipschitz-continuous such that
    \[
        |F(x) - F(z)| \leq C(1 + |x|^{|d|-n-1}) |x - z| \qquad \text{if $|x - z| \leq 1$}
    \]
    for some $C$ depending only on the function $F$. We then have
    \begin{align*}
        |\chi_{\lambda B} *\rho_\epsilon(x) F(x) - (\chi_{\lambda B} F) * \rho_\epsilon(x)| &= \left| \int_{\lambda B} \rho_\epsilon(x - z) (F(x) - F(z)) \, dz \right| \\
        &\leq \int_{\R^n} \rho_\epsilon(x - z) |F(x) - F(z)| \, dz \\
        &\leq C(1 + |x|^{|d|-n-1})\int_{\R^n} \rho_\epsilon(x - z) |x - z| \, dz\\
        &= C \epsilon (1 + |x|^{|d|-n-1}) \int_{\R^n} \rho(z) |z| \, dz.
    \end{align*}
    Moreover since both $(\chi_{\lambda B} * \rho_\epsilon) F$ and $(\chi_{\lambda B} F) *\rho_\epsilon$ are supported on $|x| \leq \lambda + \epsilon$,
    \[
        |N_\epsilon(\lambda) - \tilde N_\epsilon(\lambda)| \lesssim \epsilon \sum_{\substack{m \in \Z^n + y \\ |m| \leq \lambda + \epsilon}} (1 + |x|^{|d|-n-1}) = O(\epsilon \lambda^{|d|-1}).
    \]
\end{proof}

By the Poisson summation formula,
\begin{align*}
    \tilde E_\epsilon(\lambda) &= \sum_{m \in \Z^n + y} (\chi_{\lambda B} F)*\rho_\epsilon(m) - C_d \lambda^{|d|}\\
    &= \sum_{m \in \Z^n} e^{2\pi i\langle y,m \rangle} \widehat{\chi_{\lambda B} F}(m) \widehat \rho(\epsilon m) - C_d \lambda^{|d|}\\
    &= \sum_{m \in \Z^n \setminus 0} e^{2\pi i\langle y,m \rangle} \widehat{\chi_{\lambda B} F}(m) \widehat \rho(\epsilon m)\\
    &\hspace{8em} + \int_{\R^n} (\chi_{\lambda B} F)*\rho_\epsilon(x) \, dx - \int_{\R^n} \chi_{\lambda B}(x) F(x) \, dx.
\end{align*}
The difference of the integrals on the last line is bounded by
\begin{equation}\label{poisson main term}
    \int_{\R^n} \int_{\R^n} \rho_\epsilon(x-z) |\chi_{\lambda B}(x) F(x) - \chi_{\lambda B}(z) F(z) | \, dz \, dx
\end{equation}
We cut the outer integral into two domains, $|x| \leq \lambda - \epsilon$ and $\lambda - \epsilon < |x| \leq \lambda$. The former contributes
\begin{multline*}
    \int_{|x| \leq \lambda - \epsilon} \int_{\R^n} \rho_\epsilon(x-z) |F(x) - F(z) | \, dz \, dx\\
    \lesssim \epsilon \int_{|x| \leq \lambda - \epsilon} (1 + |x|)^{|d|-n-1} \, dx = O(\epsilon \lambda^{|d|-1})
\end{multline*}
by the same argument as in the proof of Lemma \ref{commutator}. If $\lambda - \epsilon < |x| \leq \lambda$, the inner integral in \eqref{poisson main term} is $O(\lambda^{|d|-n})$. Hence, the latter contributes
\begin{multline*}
    \int_{\lambda - \epsilon < |x| \leq \lambda} \int_{\R^n} \rho_\epsilon(x-z) |\chi_{\lambda B}(x) F(x) - \chi_{\lambda B}(z) F(z) | \, dz \, dx\\
    \lesssim \int_{\lambda - \epsilon < |x| \leq \lambda} \lambda^{|d|-n} \, dx = O(\epsilon \lambda^{|d|-1}),
\end{multline*}
and so $\eqref{poisson main term}$ is $O(\epsilon \lambda^{|d|-1})$. We are finally left with
\[
    \tilde E_\epsilon(\lambda) = \sum_{m \in \Z^n \setminus 0} e^{2\pi i\langle y,m \rangle} \widehat{\chi_{\lambda B} F}(m) \widehat \rho(\epsilon m) + O(\epsilon \lambda^{|d|-1}).
\]
Since $F$ is homogeneous of degree $|d|-n$, we have
\[
    \widehat{\chi_{\lambda B} F}(m) = \lambda^{|d|} \widehat{\chi_B F}(\lambda m)
\]
Hence, the following proposition will complete our proof.

\begin{proposition} \label{poisson sum proposition}
With everything as above,
\[
    \lambda^{|d|} \sum_{m \in \Z^n \setminus 0} |\widehat{\chi_{B} F}(\lambda m)| |\widehat \rho(\epsilon m)| = O(\epsilon^{-\frac{n-1}{2}} \lambda^{|d| - \frac{n+1}{2}}).
\]
\end{proposition}

The Proposition finally yields
\[
    \tilde E_\epsilon(\lambda) = O(\epsilon \lambda^{|d|-1} + \epsilon^{-\frac{n-1}{2}} \lambda^{|d|-\frac{n+1}{2}}).
\]
As noted previously, we optimize by setting $\epsilon = \lambda^{-\frac{n-1}{n+1}}$.
The proposition will hinge on the estimates for $|\widehat{\chi_B F}|$, below.

\begin{lemma} \label{fourier estimate}
    Let $Q = \{ \xi \in \R^n : \sup_{i}|\xi_i| \leq 1 \}$. For all real $R \geq 1$,
    \[
        \lambda^{|d|} \sum_{m \in \Z^n \cap R Q \setminus 0} |\widehat{\chi_B F}(\lambda m)| = O(R^{\frac{n-1}{2}}\lambda^{|d|-\frac{n+1}{2}}).
    \]
\end{lemma}

Since $\widehat \rho$ is Schwartz, we bound it by $|\widehat \rho(\xi)| \leq C_N\min(1,|\xi|^{-N})$ for a suitably large $N$. Assuming the lemma, we use a diadic decomposition to write the sum in Proposition \ref{poisson sum proposition} as
\begin{align*}
    &\lambda^{|d|} \sum_{m \in Z^n \cap \frac{1}{\epsilon} Q \setminus 0} |\widehat{\chi_B F}(\lambda m)| |\widehat \rho(\epsilon m)| + \lambda^{|d|} \sum_{j = 0}^\infty \sum_{m \in Z^n \cap \frac{2^j}{\epsilon} (2Q \setminus Q)} |\widehat{\chi_B F}(\lambda m)||\widehat{\rho}(\epsilon m)| \\
    &\lesssim \lambda^{|d|} \sum_{m \in Z^n \cap \frac{1}{\epsilon} Q \setminus 0} |\widehat{\chi_B F}(\lambda m)| + \lambda^{|d|} \sum_{j=0}^\infty 2^{-Nj} \sum_{m \in Z^n \cap \frac{2^j}{\epsilon} (2Q \setminus Q)} |\widehat{\chi_B F}(\lambda m)|\\
    &\lesssim \epsilon^{-\frac{n-1}{2}} \lambda^{|d| - \frac{n+1}{2}}
\end{align*}
The proof of Lemma \ref{fourier estimate} is all that remains.

\begin{proof}
Let $\beta \in C_0^\infty(\R)$ with $\beta \equiv 1$ on $[-1,1]$ and $\supp \beta \subset [-2,2]$. We write
\[
    \tilde F(x) = F(x) \prod_{i = 1}^n \beta(x_i) = \prod_{i = 1}^n \beta(x_i) \begin{cases}
        x_i^{d_i - 1} \chi_{[0,\infty)}(x_i) & \text{if $i \leq k$,} \\
        1 & \text{ if $i > k$.}
    \end{cases}
\]
Integration by parts twice in each of the $x_i$ variables yields a bound
\begin{equation} \label{F tilde bound}
    |\widehat{\tilde F}(\xi)| \lesssim \prod_{i=1}^n \langle \xi_i \rangle^{-2}.
\end{equation}
Note $\chi_B F = \chi_B \tilde F$ and hence $\widehat{\chi_B F} = \widehat \chi_B * \widehat{\tilde F}$. Using the well-known fact that
\[
    |\widehat \chi_B(\xi)| \lesssim \langle \xi \rangle^{-\frac{n+1}{2}},
\]
we write
\begin{align}
    \nonumber \sum_{m \in \Z^n \cap R Q \setminus 0} |\widehat{\chi_B F}(\lambda m)| &\lesssim \sum_{m \in \Z^n \cap RQ \setminus 0} \int_{\R^n} \langle \lambda m - \eta\rangle^{-\frac{n+1}{2}} \prod_{i = 1}^n  \langle \eta_i \rangle^{-2} \, d\eta \\
    \nonumber &= \int_{2\lambda R Q} \prod_{i = 1}^n \langle \eta_i \rangle^{-2}  \sum_{m \in \Z^n \cap RQ \setminus 0}  \langle\lambda m - \eta\rangle^{-\frac{n+1}{2}} \, d\eta \\
    \label{two integrals} &\qquad + \int_{\R^n \setminus 2\lambda R Q} \prod_{i = 1}^n \langle \eta_i \rangle^{-2}  \sum_{m \in \Z^n \cap RQ \setminus 0}  \langle \lambda m - \eta\rangle^{-\frac{n+1}{2}} \, d\eta.
\end{align}
If $\eta \in 2\lambda RQ$, then by the integral test
\begin{align*}
    \sum_{m \in \Z^n \cap RQ \setminus 0} \langle \lambda m - \eta \rangle^{-\frac{n+1}{2}} &\lesssim \int_{RQ} |\lambda \xi - \eta|^{-\frac{n+1}{2}} \, d\xi \\
    &\leq \lambda^{-\frac{n+1}{2}}\int_{3RQ} |\xi|^{-\frac{n+1}{2}}\, d\xi \lesssim \lambda^{-\frac{n+1}{2}} R^{\frac{n-1}{2}}.
\end{align*}
Hence, the first integral in the last line of \eqref{two integrals} is bounded by
\[
    R^{\frac{n-1}{2}} \lambda^{-\frac{n+1}{2}} \int_{2\lambda RQ} \prod_{i = 1}^n \langle \eta_i \rangle^{-2} \, d\eta \lesssim  R^{\frac{n-1}{2}} \lambda^{-\frac{n+1}{2}}.
\]
On the other hand if $\eta \not\in 2\lambda R Q$ and $m \in RQ$, then $\langle\lambda m - \eta\rangle \approx |\eta|$ and so the second of the integrals in \eqref{two integrals} is bounded by
\begin{align*}
    &R^n \int_{\R^n \setminus 2\lambda R Q} |\eta|^{-\frac{n+1}{2}} \prod_{i = 1}^n \langle \eta_i \rangle^{-2} \, d\eta \\
    &\leq R^n \sum_{j = 1}^n \int_{\substack{|\eta_j| = \max_i |\eta_i| \\ |\eta_j| \geq 2\lambda R}} |\eta|^{-\frac{n+1}{2}} \prod_{i = 1}^n \langle \eta_i \rangle^{-2} \, d\eta \\
    &\leq R^n \sum_{j = 1}^n \int_{\substack{|\eta_j| = \max_i |\eta_i| \\ |\eta_j| \geq 2\lambda R}} |\eta_j|^{-\frac{n+5}{2}} \prod_{i \neq j} \langle \eta_i \rangle^{-2} \, d\eta \\
    &= R^n \sum_{j = 1}^n \int_{|\eta_j| \geq 2\lambda R} |\eta_j|^{-\frac{n+5}{2}} \left( \prod_{i \neq j} \int_{|\eta_i| \leq |\eta_j|} \langle \eta_i \rangle^{-2} \, d\eta_i \right) \, d\eta_j\\
    &\lesssim R^n \sum_{j = 1}^n \int_{|\eta_j| \geq 2\lambda R} |\eta_j|^{-\frac{n+5}{2}} \, d\eta_j \\
    &\lesssim R^{\frac{n-3}{2}} \lambda^{-\frac{n+3}{2}},
\end{align*}
which is better than the bound on the first integral. The lemma follows.
\end{proof}


\section{Proof of Theorem \ref{DG products}}

A set of measure zero in an open neighborhood of $\R^d$ remains a set of measure zero after mapping it through a diffeomorphism. Hence we say a subset of a smooth manifold has measure zero if it has measure zero in local coordinates. Note we do not require any particular measure on the manifold, only a smooth structure. If our manifold is a bundle, e.g. $T^*M$ or $S^*M$, then Fubini's theorem in local coordinates tells us a subset has measure zero if and only if its intersection with almost every fiber has measure zero in the fiber.

We use the following notation for general Riemannian manifolds $M$. We use $(x,\xi)$ to denote an element in $T^*M$ in canonical local coordinates, where $(x,\xi)$ projects onto $x \in M$. We begin by considering the Hamilton flow associated with a slightly different symbol than in \eqref{principal symbol}. Take 
\begin{equation} \label{principal symbol 2}
    \tilde p(x,\xi) = \frac{1}{2} \sum_{i,j} g_{ij}(x) \xi_i \xi_j.
\end{equation}
Note the restrictions of $H_{\tilde p}$ and $H_p$ to $S^*M = \{(x,\xi) \in T^*M : p(x,\xi) = 1\}$ coincide. So, as far as the theorem is concerned, we may substitute $H_{p}$ with $H_{\tilde p}$. Note $H_{\tilde p}$ is homogeneous of degree $1$ on $T^*M$, where $H_p$ is homogeneous of degree $0$. The flow associated to $H_{\tilde p}$ is sometimes called the ``geometer's geodesic flow" and is denoted $e^{tH_{\tilde p}}$. Since $H_{\tilde p}$ is homogeneous of degree $1$,
\[
    e^{tH_{\tilde p}}(x,s\xi) = se^{stH_{\tilde p}}(x,\xi).
\]
We will be concerned with the flow $e^{H_{\tilde p}}$, the flow along $H_{\tilde p}$ after time $1$. Set
\begin{equation}\label{sigma p}
    \Sigma_{\tilde p} = \{ (x,\xi) \in T^*M : e^{H_{\tilde p}}(x,\xi) = (x,\xi) \}.
\end{equation}
If $\pi : T^*M \setminus 0 \to S^*M$ is the natural projection onto the unit sphere bundle, then $\pi(\Sigma \setminus 0)$ is the set of directions in $S^*M$ which belong to periodic orbits.

Now assume the hypotheses of the theorem. Since $M = M_1 \times \cdots \times M_n$ is a product manifold, the metric reads
\[
    g = \begin{bmatrix}
        g_{M_1} & 0 & \cdots & 0 \\
        0 & g_{M_2} & \cdots & 0 \\
        \vdots & \vdots & \ddots & \vdots \\
        0 & 0 & \cdots & g_{M_n}
    \end{bmatrix}
\]
and hence
\[
    \tilde p(x,\xi) = \sum_{i = 1}^n \tilde p_{i}(x_i, \xi_i)
\]
where $(x,\xi) = ((x_1,\xi_1),\ldots,(x_n,\xi_n))$ with $(x_i,\xi_i) \in T^*M_i$ for each $i = 1,\ldots,n$, and where $\tilde p_i$ is defined as in \eqref{principal symbol 2} for $M_i$. The Hamilton vector field on $M$ then splits into a direct sum
\[
    H_{\tilde p} = H_{\tilde p_1} \oplus \cdots \oplus H_{\tilde p_n}.
\]
The Hamilton flow $e^{tH_{\tilde p}}$ splits similarly,
\[
    e^{tH_{\tilde p}} (x,\xi) = (e^{tH_{\tilde p_1}} (x_1,\xi_1), \ldots, e^{tH_{\tilde p_n}} (x_n,\xi_n) ).
\]
It immediately follows that
\begin{equation} \label{sigma p lemma}
    \Sigma_{\tilde p} = \bigoplus_{i = 1}^n \Sigma_{\tilde p_i}
\end{equation}
where $\Sigma_{\tilde p} \subset T^*M$ and $\Sigma_{\tilde p_i} \subset T^*M_i$ for $i = 1,\ldots,n$ are as in \eqref{sigma p}.
The following lemma reveals an integer lattice hidden in $\Sigma_{\tilde p}$.

\begin{lemma}\label{hidden lattice}
For each $i = 1,\ldots,n$, fix $(x_i,\xi_i) \in \Sigma_i \setminus 0$ and consider their span
\[
    V = \bigoplus_{i = 1}^n \operatorname{span}(x_i,\xi_i)
\]
in $T_x^*M$. Then, $\Sigma_{\tilde p} \cap V \simeq \Z^n$.
\end{lemma}

\begin{proof}
Assume without loss of generality that for each $i$, $(x_i,t \xi_i) \not\in \Sigma_{\tilde p_i}$ for $0 < t < 1$. Suppose $(x,\eta) \in V$, that is
\[
    (x,\eta) = ((x_1,t_1\xi_1),\ldots,(x_n,t_n\xi_n))
\]
for some real coefficients $t_1,\ldots, t_n$. If $(x,\eta) \in \Sigma_{\tilde p}$, then $(x_i,t_i\xi_i) \in \Sigma_{\tilde p_i}$ for each $i$ by \eqref{sigma p lemma}. But, $(x_i,t\xi_i) \in \Sigma_{\tilde p_i}$ if and only if $t_i \in \Z$. This yields a bijection 
\begin{align*}
    \Z^n &\to \Sigma_{\tilde p} \cap V \\
    (t_1,\ldots, t_n) &\mapsto ((x_1,t_1\xi_1),\ldots,(x_n,t_n\xi_n)),
\end{align*}
as required.
\end{proof}

Note almost every $n$-dimensional subspace $V$ of $T^*_xM$ can be written as the span of $(x_i,\xi_i) \in T^*_{x_i}M_i \setminus 0$ for $i = 1,\ldots,n$. Moreover if for some $i$, $\operatorname{span}(x_i,\xi_i)$ does not intersect $\Sigma_{\tilde p_i} \setminus 0$, then $\Sigma_{\tilde p} \cap V$ is contained completely in a lower-dimensional subspace of $V$, and hence has measure zero. The lemma then ensures that $\Sigma_{\tilde p} \cap V$ has measure zero in $T^*_xM$ for almost every $n$-dimensional subspace $V$. The following well-known fact about integration on Grasmannians allows us to conclude that $\Sigma_{\tilde p}$ has measure zero in $T^*_x M$.

\begin{lemma} \label{measure zero subset}
    Fix natural numbers $d$ and $n < d$. Let $\Gr(n,d)$ denote the Grassmannian, the manifold of all $n$-dimensional subspaces of $\R^d$ equipped with a measure $\mu$ invariant under orthogonal transformations. Then for all $f \in L^1(\R^n)$,
    \[
        \int_{\Gr(n,d)} \left( \int_V |x|^{d-n} f(x) \, d\sigma_V(x) \right) \, d\mu(V) = C\int_{\R^d} f(x) \, dx,
    \]
    where $\sigma_V$ is the restriction measure to $V \subset \R^d$. In particular, if $\sigma_V(E \cap V) = 0$ for almost every $V \in \Gr(n,d)$, then $E$ has measure zero in $\R^d$.
\end{lemma}

Let $\R_+ \Sigma_{\tilde p} \setminus 0$ denote the set of rays in the fibers of $T^*M$ originating at $0$ and intersecting $\Sigma_{\tilde p}\setminus 0$. Note since $V$ has dimension $n \geq 2$, $(\R_+ \Sigma_{\tilde p} \setminus 0) \cap V$ is a measure zero subset of $V$. Lemma \ref{measure zero subset} allows us to conclude that $\R_+ \Sigma_{\tilde p} \setminus 0$ is a set of measure zero in each of the fibers of $T^*M$. We conclude that $\pi(\Sigma_{\tilde p} \setminus 0) = \pi(\R_+ \Sigma_{\tilde p} \setminus 0)$ has measure zero in each of the fibers of $S^*M$, and hence has measure zero in $S^*M$.


\section{Proof of Proposition \ref{not a spectrum}}

For each integer $k = 1,2,\ldots$, we add to $\Lambda$ the integers lying in $[2^{k-1/2}+1,2^{k}-1]$. Furthermore, we include 
\[
    \sqrt{2^{2k} - n^2} \qquad n \in \Z \cap [2^{k-1/2}+1,2^k-1].
\]
We point out two things. First, that $\sqrt{2^{2k} - n^2}$ lies in the interval $(2^{k-1}, 2^{k-1/2})$ which is removed by a distance of $1$ from both $[2^{k-1/2}+1,2^k-1]$ and $[0,2^{k-1}-1]$. Second, that the gaps between the successive $\sqrt{2^{2k} - n^2}$ are greater than $1$. Hence each unit interval $[n,n+1)$ for $n = 0,1,2,\ldots$ contains at most one point assigned to $\Lambda$ in this way. We complete $\Lambda$ by adding a point wherever there is an empty unit interval.

Let
\[
    N(\lambda) = \#(\Lambda \times \Lambda) \cap \lambda B_2
\]
count the number of points in $\Lambda \times \Lambda$ lying in the closed disk of radius $\lambda$. By comparing the area of the disk with a union of unit squares, we have 
\[
    N(\lambda) = \pi \lambda^2 + O(\lambda).
\]
However for each $k = 1,2,\ldots$,
\[
    \#\{(\lambda_1,\lambda_2) \in \Lambda \times \Lambda : \lambda_1^2 + \lambda_2^2 = 2^{2k}\} \geq 2\#\Z \cap [2^{k-1/2}+1,2^{k}-1]
\]
by construction. Hence
\[
    N(2^k + \delta) - N(2^k - \delta) \geq 2^k (2 - \sqrt{2}) - 3,
\]
and
\[
    \limsup_{\lambda \to \infty} \frac{|N(\lambda) - \pi \lambda^2|}{\lambda} \geq (1 - 1/\sqrt2).
\]


\bibliography{references}{}
\bibliographystyle{alpha}

\end{document}